\documentclass[final,1p,times]{elsarticle}

\usepackage{amssymb}
\usepackage{amsthm}
\usepackage{amscd}
\usepackage{bm}
\usepackage{amsmath}
\usepackage{amsfonts}
\usepackage{graphicx}
\newtheorem{theorem}{Theorem}[section]
\newtheorem{remark}[theorem]{Remark}
\newtheorem{example}[theorem]{Example}

\newtheorem{lemma}{Lemma}[section]

\usepackage{mathrsfs}
\usepackage{titletoc}
\usepackage{color}

\newcommand{\be}{\begin{equation}}

\newcommand{\ee}{\end{equation}}
\newcommand{\bea}{\begin{eqnarray}}
\newcommand{\eea}{\end{eqnarray}}
\newcommand{\bna}{\begin{eqnarray*}}
\newcommand{\ena}{\end{eqnarray*}}

\usepackage{float,amssymb,mathrsfs,bm,multirow,graphics,fancyhdr}

\usepackage{geometry}
\journal{***}

\begin{document}

\begin{frontmatter}

\title{A note for $W^{1,p}(V)$ and $W_{0}^{1,p}(V)$ on a locally finite graph}

\author[bnu]{Yulu Tian\corref{Tian}}
\ead{tianyl@mail.bnu.edu.cn}
\author[bnu]{Liang Zhao}
\ead{liangzhao@bnu.edu.cn}
\address[bnu]{School of Mathematical Sciences, Key Laboratory of Mathematics and Complex Systems of MOE,\\
	Beijing Normal University, Beijing, 100875, China}

\cortext[Tian]{Corresponding author.}

\begin{abstract}
  In this paper, we investigate the Sobolev spaces $W^{1,p}(V)$ and $W_{0}^{1,p}(V)$ on a locally finite graph $G=(V,E)$, which are fundamental tools when we apply the variational methods to partial differential equations on graphs. As a key contribution of this note, we show that in general, $W^{1,p}(V)\neq W_0^{1,p}(V)$ on locally finite graphs, which is different from the situation on Euclidean space $\mathbb{R}^N$.
\end{abstract}


\end{frontmatter}

\section{Introduction}\label{sec1}

In addressing a partial differential equation (PDE) by seeking the critical points of the associated energy functional, it is imperative to operate within a suitable Sobolev space to guarantee the functional's well-definedness. The selection of the Sobolev space is, therefore, pivotal in the calculus of variations applied to the equation. For example, given that $W^{1,2}(\mathbb{R}^N)$ is reflexive and $C_0^{\infty}(\mathbb{R}^N)$ is dense in it, we may utilize $\phi\in C_0^{\infty}(\mathbb{R}^N)$ as test functions to characterize a weak solution in $W^{1,2}(\mathbb{R}^N)$ for a nonlinear Schr\"{o}dinger equation defined on the Euclidean space $\mathbb{R}^N$. Moreover, the regularity of the weak solution is contingent upon the embedding characteristics of the Sobolev space and the nonlinear structure of the equation.

PDEs on graphs extend the classical study of PDEs from Euclidean spaces or manifolds to graph structures. Since graphs are capable of modeling phenomena in various fields, such as mathematical physics, network analysis, and computer science, this area has witnessed considerable advancement in recent years \cite{chang2023ground,grigor2016kazdan,grigor2016yamabe,grigor2017existence,han2020existence,hou2022existence,hua2023class,hua2023existence,keller2018kazdan,sun2022brouwer,zhang2018convergence}. Specifically, when exploring variational methods for PDEs on graphs, comprehending Sobolev spaces within the graph framework remains essential. In this paper, our primary focus is on the spaces $W^{1,p}(V)$ and $W_0^{1,p}(V)$, as their attributes are instrumental in identifying potential solutions to kinds of important nonlinear equations on graphs. To delineate the specific problems we address, we begin by introducing some foundational notations for graphs.

Consider a connected graph $G=(V,E)$ endowed with a measure $\mu: V\to \mathbb{R}^{+}$ defined on its vertices and a symmetric weight function $w: E\to \mathbb{R}^{+}$ assigned to its edges.  A commonly used norm for $u\in W^{1,p}(V)$ is
\begin{equation*}
	\Vert u\Vert_{W^{1,p}(V)}^p=\|u\|_p^p+\|\nabla u\|_p^p, \ \ 1\leq p<+\infty,
\end{equation*}
where $\|u\|_p^p=\underset{x\in V}\sum\mu_x |u(x)|^p$ and $\|\nabla u\|_p^p=\underset{x\in V}\sum\mu_x\left(\frac{1}{2\mu_x}\underset{xy\in E}\sum w_{xy}\left\vert u(y)-u(x)\right\vert^2\right)^{p/2}$, $\mu_x$ is the measure on the vertex $x\in V$ and $w_{xy}$ is the weight on the edge $xy\in E$. The space $W_0^{1,p}(V)$ is defined as the completion of the space $C_c(V)$ with respect to the given norm, where $C_c(V)$ denotes the set of functions that are compactly supported on $V$. In the context of PDEs on graphs, as discussed in \cite{ge2018p,ge2020p,zhang2019positive}, one may alternatively employ the space $\mathcal{W}^{1,p}(V)$, equipped with the norm
\begin{equation*}
	\Vert u\Vert_{\mathcal{W}^{1,p}(V)}^p=\|u\|_p^p+\|\nabla_{\mathcal{W}} u\|_p^p,
\end{equation*}
where $\|\nabla_{\mathcal{W}} u\|_p^p=\frac{1}{2}\underset{x\in V}\sum \underset{xy\in E}\sum w_{xy}\vert u(y)-u(x)\vert^p$. It is evident that the two norms  $\|\cdot\|_{W^{1,p}}$ and $\Vert \cdot\Vert_{\mathcal{W}^{1,p}}$ coincide when $p=2$, yet they differ for $p\neq 2$.

The structure of a graph can introduce nuances that make the Sobolev spaces on the graph distinct from that in the continuous setting. In the context of finite graphs, the Sobolev space $W^{1,2}(V)$, along with its embeddings, has been explored in a sequence of papers by Grigor'yan et al. \cite{grigor2016kazdan,grigor2016yamabe,grigor2017existence}. The higher order space $W^{2,2}(V)$ and its embedding on locally finite graphs have been studied by Han et al. \cite{han2020existence}. These studies generally require a positive lower bound condition on the measure $\mu$, and sometimes, a positive upper bound on the weight $w$ is also postulated. More recently, Shao et al. \cite{shao2023sobolev} have established several core properties of the Sobolev spaces $W^{m,p}(V)$ and $W_{0}^{m,p}(V)$ on connected and locally finite graphs. Their findings, which demonstrate completeness, reflexivity, and separability, are notable for not being contingent on the bounds of $\mu$ and $w$. In this paper, they also raised an interesting question:\\

{\it Are the two spaces $W^{1,p}(V)$ and $W_0^{1,p}(V)$ exactly the same?}\\

In the Euclidean space, the Sobolev spaces $W^{1,p}(\mathbb{R}^N)$ and $W_{0}^{1,p}(\mathbb{R}^N)$ coincide for $1\leq p<+\infty$. This equivalence has been explored in several previous works on graphs under various assumptions, with affirmative conclusions. For a connected and locally finite graph $G=(V,E)$, suppose there exists a constant $\mu_{\min}>0$ such that for every $x\in V$, $\mu_x\geq \mu_{\min}$, and the degree $\deg_{x}:=\underset{xy\in E}\sum w_{xy}$ is uniformly bounded above by a constant $C>0$. Under these conditions, the results in \cite[Propositions 2.1 and 2.2]{han2020existence} assert that $W^{1,2}(V)$ is the completion of $C_c(V)$. With the same conditions on measure and degree, Shao et al. in \cite[Proposition 6.2]{shao2023sobolev} established that $W^{1,p}(V)$ is the completion of $C_c(V)$ for $1\leq p<+\infty$. Furthermore, assuming $\underset{x\in V}\sup \frac{\deg_x}{\mu_x}<+\infty$, Shao and Han demonstrated in \cite[Proposition 5.7]{han2021p} that for $1<p<+\infty$, $W^{1,p}(V)$ is the completion of $C_c(V)$. In the significant monograph on the spectral geometry of infinite graphs \cite{keller2021graph}, there are also various criteria for $W^{1,2}(V)=W_0^{1,2}(V)$, appearing as Markov uniqueness.

Contrary to the aforementioned results, the primary contribution of this note is to present a negative answer to the question of equivalence between $W^{1,p}(V)$ and $W_0^{1,p}(V)$. Specifically, we show that these spaces are not necessarily identical if the locally finite graph $G(V,E)$ possesses a positive Cheeger constant
$$
\alpha:=\underset{\Omega\subset V}\inf\frac{\vert\partial\Omega\vert}{\vert\Omega\vert}.
$$
Here $\Omega\subset V$ is a finite vertex set, $\partial\Omega=\{y\in V\setminus\Omega: \text{there exists } x\in \Omega\text{ such that } xy\in E\}$ is the boundary of $\Omega$ and
$$
\vert\Omega\vert=\sum_{x\in \Omega}\deg_x, \ \ \vert\partial\Omega\vert=\sum_{y\in\partial\Omega}\deg_y.
$$
Precisely, we prove the following results for a connected and locally finite graph $G=(V,E)$.

\begin{theorem}\label{pgeq2}
	For $p\geq 2$, we have $W^{1,p}(V)\neq W_{0}^{1,p}(V)$ if $w_{xy}\equiv1$, $\underset{x\in V}\max\mu_x\leq1$, $|V|=\sum_{x\in V}\mu_x<+\infty$ and the Cheeger constant $\alpha>0$.
\end{theorem}

If we consider the spaces $\mathcal{W}^{1,p}(V)$ and $\mathcal{W}^{1,p}_0(V)$, where $\mathcal{W}^{1,p}_0(V)$ is the completion of $C_c(V)$ under the norm $\Vert\cdot\Vert_{\mathcal{W}^{1,p}(V)}$, it holds that the following theorem is true.
\begin{theorem}\label{pg1}
	Under the assumptions of Theorem \ref{pgeq2}, we have $\mathcal{W}^{1,p}(V)\neq \mathcal{W}_{0}^{1,p}(V)$ for $p\geq1$.
\end{theorem}

\section{Proof of the theorems}\label{sec4}

We base our main idea on the Cheeger inequality as presented in \cite{keller2016general}. It should be pointed out that the case $p=2$ in our theorems has been discussed in \cite[Section 4]{keller2012dirichlet}, from the perspective of Dirichlet forms. First we present three useful lemmas found in \cite{bauer2015cheeger,keller2016general}, the proofs of which are omitted for brevity.
\begin{lemma}\label{area}
	(Area formula). For a nonnegative function $f\in C(V)$ satisfying $\underset{x\in V}\sum \deg_x |f(x)|<+\infty$, there holds
	\begin{equation*}
		\sum_{x\in V}\deg_x f(x)=\int_{0}^{+\infty} |\Omega_{t}(f)|dt,
	\end{equation*}
	where $\Omega_{t}(f):=\{x\in V:f(x)>t\}$.
\end{lemma}

\begin{lemma}\label{coarea}
	(Co-area formula). Under the same assumptions as in Lemma \ref{area}, there holds
	\begin{equation*}
		\frac{1}{2}\sum_{x\in V}\sum_{xy\in E}w_{xy}\vert f(y)-f(x)\vert=\int_{0}^{+\infty}\vert\partial\Omega_t(f)\vert dt.
	\end{equation*}
\end{lemma}

\begin{lemma}\label{fp}
	For any constants $a,b>0$ and $p\geq 1$, we have
	\begin{equation*}
		\vert a^p-b^p\vert\leq p\left(\frac{a^p+b^p}{2}\right)^\frac{p-1}{p}\vert a-b\vert.
	\end{equation*}
\end{lemma}

The following lemma presents a Cheeger inequality for $p\geq 1$. For convenience in applying the inequality, we adopt notations that are commonly used in the variational methods on graphs to describe and prove it. For $\phi\in C_c(V)$, we define a $p$-energy $\mathcal{E}_p:C_c(V)\to[0,+\infty)$ by
\begin{equation*}\label{penergy}
	\mathcal{E}_p(\varphi):=\frac{1}{2}\sum_{x\in V}\sum_{xy\in E}w_{xy}\vert \varphi(y)-\varphi(x)\vert^p.
\end{equation*}
Additionally, we require the quantity
\begin{equation*}
	\lambda_p:=\inf_{0\neq\varphi\in C_c(V)}\frac{\mathcal{E}_p(\varphi)}{\underset{x\in V}\sum \deg_x |\varphi(x)|^p}.
\end{equation*}
We can then state the following Cheeger inequality:
\begin{lemma}\label{pcheeger}
	For $p\geq 1$, the following inequality holds:
	\begin{equation*}
		\frac{2^{p-1}}{p^p}\alpha^p\leq\lambda_p.
	\end{equation*}
\end{lemma}

\begin{proof} For all $\varphi\in C_c(V)$, taking $f=\vert\varphi\vert^p$ in Lemmas \ref{area} and \ref{coarea} and using Lemma \ref{fp}, we obtain
	
	\begin{equation*}
		\begin{aligned}
			\alpha \sum_{x\in V}\deg_x |\varphi(x)|^p&=\alpha\int_{0}^{+\infty} \vert\Omega_t( |\varphi|^p)\vert dt\leq\int_{0}^{+\infty}\vert\partial\Omega_t(\vert \varphi\vert^p)\vert dt\\
			&=\frac{1}{2}\sum_{x\in V}\sum_{xy\in E}w_{xy}\vert\left(\vert\varphi(y)\vert^p-\vert\varphi(x)\vert^p\right)\vert\\
			&\leq\frac{p}{2}\sum_{x\in V}\sum_{xy\in E}w_{xy}\left(\frac{\vert\varphi(y)\vert^p+\vert\varphi(x)\vert^p}{2}\right)^\frac{p-1}{p}\vert\varphi(y)-\varphi(x)\vert\\
			&\leq\frac{p}{2}\sum_{x\in V}\sum_{xy\in E}\left(\frac{w_{xy}}{2}\left(\vert\varphi(y)\vert^p+\vert\varphi(x)\vert^p\right)\right)^\frac{p-1}{p}{w_{xy}}^{\frac{1}{p}}\vert\varphi(y)-\varphi(x)\vert\\
			&\leq\frac{p}{2}\left(\sum_{x\in V}\deg_x\vert\varphi(x)\vert^p\right)^\frac{p-1}{p}\left(\sum_{x\in V}\sum_{xy\in E}w_{xy}\vert\varphi(y)-\varphi(x)\vert^p\right)^{1/p}\\
			&=\frac{p}{2}\left(\sum_{x\in V}\deg_x\vert\varphi(x)\vert^p\right)^\frac{p-1}{p}\left(2\mathcal{E}_p(\varphi)\right)^{1/p},
		\end{aligned}
	\end{equation*}
	where the fourth inequality is due to H\"{o}lder's inequality. As a consequence, we can immediately obtain the desired result.
\end{proof}

To prove our theorems, we also need the following inequality:
\begin{lemma}\label{cp}
	Assume that $p\geq 1$ and $a_i\geq0$, $i=1,2,\ldots,n$. Then we have
	\begin{equation*}
		\sum_{i=1}^{n}a_i^p\leq\left(\sum_{i=1}^{n}a_i\right)^p.
	\end{equation*}
\end{lemma}
\begin{proof}
	We prove the lemma by induction. Consider the function $f(x):=(x+1)^p-x^p-1$, which is increasing for $x \in [0, +\infty)$. This implies that
	\begin{equation*}
		a^p+b^p\leq(a+b)^p.
	\end{equation*}
	The inequality obviously holds for $n=2$. Assume it holds for some $n > 2$, i.e., $\sum\limits_{i=1}^{n-1}a_i^p \leq \left(\sum\limits_{i=1}^{n-1}a_i\right)^p$. Then, for $n$ terms, we have
	\begin{equation*}
		a_1^p+a_2^p+\cdots+a_n^p\leq a_1^p+\left(a_2+\cdots+a_n\right)^p\leq\left(a_1+a_2+\cdots+a_n\right)^p,
	\end{equation*}
	which completes the proof.
\end{proof}

\indent\textbf{Proof of Theorem~{\upshape\ref{pgeq2}}}\quad
	Due to the fact that $|V|<+\infty$, we have $\varphi_0\equiv 1\in W^{1,p}(V)$. If $W^{1,p}(V)=W^{1,p}_0(V)$, then there exists a sequence $\{\varphi_k\}\subset C_c(V)$ such that $\varphi_k\to \varphi_0$ in $W^{1,p}(V)$ as $k\to+\infty$. Furthermore, it follows from Lemmas \ref{pcheeger} and \ref{cp} that for any $p\geq 2$, $k\geq1$ and $x_0\in V$, there holds
	\begin{equation*}
		\begin{aligned}
			\Vert\varphi_k-\varphi_0\Vert^p_{W^{1,p}(V)}&\geq\Vert\nabla(\varphi_k-\varphi_0)\Vert_p^p=\Vert\nabla\varphi_k\Vert_p^p\\
			&=\sum_{x\in V}\mu_x\left(\frac{1}{2\mu_x}\sum_{xy\in E}w_{xy}\vert\varphi_k(y)-\varphi_k(x)\vert^2\right)^{p/2}\\
			&=\left(\frac{1}{2}\right)^{p/2}\sum_{x\in V}\frac{1}{\mu_x^{\frac{p}{2}-1}}\left(\sum_{xy\in E}\vert\varphi_k(y)-\varphi_k(x)\vert^2\right)^{p/2}\\
			&\geq\left(\frac{1}{2}\right)^\frac{p-2}{2}\frac{1}{2}\sum_{x\in V}\sum_{xy\in E}\vert\varphi_k(y)-\varphi_k(x)\vert^p\\
			&=\left(\frac{1}{2}\right)^\frac{p-2}{2}\mathcal{E}_p(\varphi_k)\geq\left(\frac{1}{2}\right)^\frac{p-2}{2}\lambda_p\sum_{x\in V}\deg_x\vert\varphi_k(x)\vert^p\\
			&\geq\frac{2^\frac{p}{2}}{p^p}\alpha^p_p\sum_{x\in V}\deg_x\vert\varphi_k(x)\vert^p\\
			&\geq\frac{2^\frac{p}{2}}{p^p}\alpha^p_p \deg_{x_0}\vert\varphi_k(x_0)\vert^p.
		\end{aligned}
	\end{equation*}
	Since $\varphi_k$ converges to $\varphi_0$ in $W^{1,p}(V)$, it follows that $\varphi_k(x_0)\to \varphi_0(x_0)=1$ as $k\to+\infty$. Therefore, there exists $K_0>0$ such that for all $k>K_0$, we have $\varphi_k(x_0)>\frac{1}{\sqrt{2}}$. Then as $k\to+\infty$, we consider the following inequality:
	\begin{equation*}
		0\leftarrow\Vert\varphi_k-\varphi_0\Vert^p_{W^{1,p}(V)}\geq\frac{2^\frac{p}{2}}{p^p}\alpha^p_p\left(\frac{1}{\sqrt{2}}\right)^p\deg_{x_0}\geq\frac{\alpha^p_p}{p^p}\deg_{x_0}\geq \frac{\alpha^p_p}{p^p},
	\end{equation*}
	which leads to a contradiction. Thus, we can conclude that $W^{1,p}(V)\neq W_{0}^{1,p}(V)$.
$\hfill\square$

Next, we proceed with the proof of Theorem \ref{pg1}, continuing to employ the notations established in the preceding proof.

\indent\textbf{Proof of Theorem~{\upshape\ref{pg1}}}\quad
    For any $x_0\in V$, we have
	\begin{equation*}
		\begin{aligned}
			\Vert\varphi_k-\varphi_0\Vert^p_{\mathcal{W}^{1,p}(V)}&\geq\Vert\nabla_\mathcal{W}(\varphi_k-\varphi_0)\Vert_p^p
			=\Vert\nabla_\mathcal{W}\varphi_k\Vert_p^p\\
			&=\frac{1}{2}\sum_{x\in V}\sum_{xy\in E}w_{xy}\vert\varphi_k(y)-\varphi_k(x)\vert^p\\
			&=\mathcal{E}_p(\varphi_k)\geq\lambda_p\sum_{x\in V}\deg_x\vert\varphi_k(x)\vert^p\\
			&\geq\frac{2^{p-1}}{p^p}\alpha^p_p\sum_{x\in V}\deg_x\vert\varphi_k(x)\vert^p\\
			&\geq\frac{2^{p-1}}{p^p}\alpha^p_p\vert\varphi_k(x_0)\vert^p\deg_{x_0}.
		\end{aligned}
	\end{equation*}
	The subsequent portion of the proof follows the methodology applied in the proof of Theorem \ref{pgeq2}.
$\hfill\square$

At the end of this section, we provide an example of a connected and locally finite graph that satisfies the conditions of the theorems presented.

\begin{example}\label{example}
	As depicted in Figure \ref{graph}, let $G(V,E)$ be a tree-like connected and locally finite graph. We assign the weight $w_{xy}=1$ for each $xy\in E$, and define the measure of the vertex $x_k$ in the $k$-th layer as $\mu_{x_k}=3^{-k}$. It is evident that the graph satisfies $|V|=1<+\infty$ and $\alpha=1/3>0$.
\end{example}

\begin{figure}
	\centering
	\includegraphics[height=7cm,width=7cm]{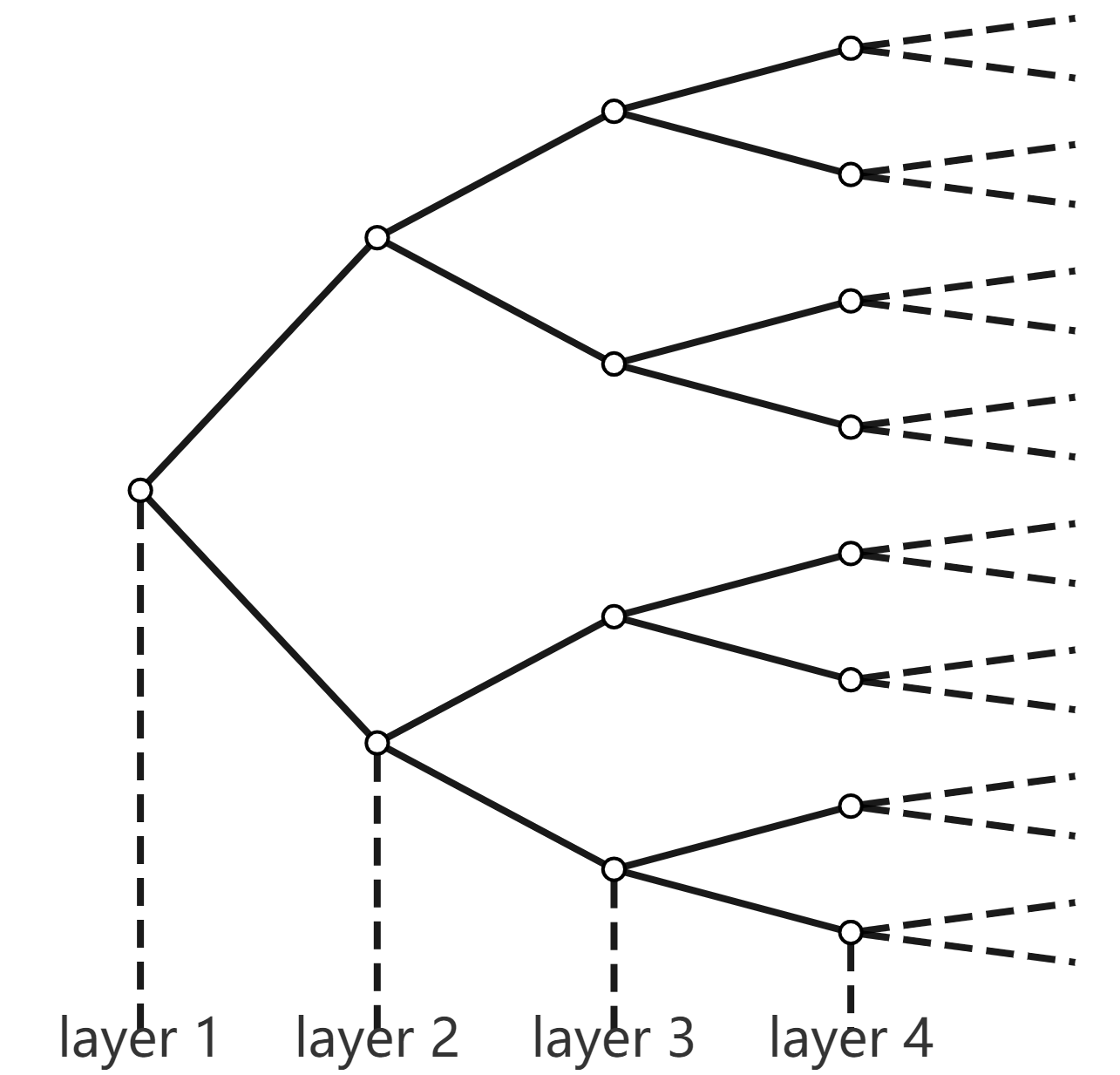}
	\caption{$G(V,E)$}
    \label{graph}
\end{figure}

\begin{remark}
	When discussing hyperbolic models with negative curvature on graphs \cite{ni2019community}, the term usually refers to the tree-like structure shown in Figure \ref{graph}. In contrast, when discussing smooth hyperbolic manifolds, the results of Aubin \cite{aubin1976espaces, hebey1996sobolev} indicate that in the complete hyperbolic space $\mathbb{H}^n$, we have $W^{1,p}(\mathbb{H}^n)=W_0^{1,p}(\mathbb{H}^n)$. Upon comparing the conclusions drawn from smooth manifolds and graphs, it is evident that a key condition for our theorems to be valid is that $|V|<+\infty$, which means that Example \ref{example} does not correspond to $\mathbb{H}^n$ whose volume is infinite.
\end{remark}

\section*{Acknowledgements}
This research is supported by the National Key R and D Program of China (2020YFA0713100), the National Natural Science Foundation of China (12271039 and 12101355) and the Open Project Program (K202303) of Key Laboratory of Mathematics and Complex Systems, Beijing Normal University. We are very grateful to Professor M. Keller for pointing out that we can use the Cheeger inequality to deal with this issue.




\end{document}